\documentclass[11pt]{amsart}
\usepackage{mathrsfs,amsmath,amssymb,amsthm}
\usepackage{comment}
\usepackage{enumerate}
\newtheorem{theorem}{Theorem}

\newtheorem{example}{Example}

\newtheorem{lemma}{Lemma}
\newtheorem{corollary}{Corollary}

\newtheorem{remark}{Remark}

\newcommand{\zz}{\mathbf z}
\newcommand{\xx}{\mathbf x}
\newcommand{\yy}{\mathbf y}
\newcommand{\ww}{\mathbf w}

\def\mapright#1{\smash{\mathop{\longrightarrow}\limits^{{#1}}}}
\def\mapdown#1{\Big\downarrow\rlap{$\vcenter{\hbox{$#1$}}$}}

\title{A generalized Join theorem for real analytic singularities}
\author{Kazumasa Inaba} 
\address{Department of Applied Science, Faculty of Science, Okayama University of Science, 
1-1 Ridai-cho, Kita-ku, Okayama-city, Okayama 700-0005, Japan}
\email{k-inaba@ous.ac.jp}

\begin{document}
\renewcommand{\thefootnote}{\fnsymbol{footnote}}
\footnote[0]{2020\textit{ Mathematics Subject Classification}.
Primary 32S55; Secondary: 58K05, 58K10.} 


\maketitle


\begin{abstract}
Let $f_{1} : (\Bbb{R}^{n}, \mathbf{0}_{n}) \rightarrow (\Bbb{R}^{2}, \mathbf{0}_{2})$ 
and $f_{2} : (\Bbb{R}^{m}, \mathbf{0}_{m}) \rightarrow (\Bbb{R}^{2}, \mathbf{0}_{2})$ 
be real analytic germs of independent variables, where $n, m \geq 2$. 
Assume that $f_{1}$ and $f_{2}$ satisfy 
$a_{f}$-condition at $\mathbf{0}_{2}$. 
Let $g$ be a $2$-variable strongly non-degenerate mixed polynomial which 
is locally tame along vanishing coordinate subspaces. 
Then a real analytic germ 
$f : (\Bbb{R}^{n} \times \Bbb{R}^{m}, \mathbf{0}_{n+m}) \rightarrow (\Bbb{R}^{2}, \mathbf{0}_{2})$ 
is defined by $f(\xx, \yy) = g(f_{1}(\xx), f_{2}(\yy))$. 

In this paper, we first show the existence of the Milnor fibration of $f$. 
We next show a generalized Join theorem for real analytic singularities. 
By this theorem, the homotopy type of the Milnor fiber of $f$ is determined 
by those of $f_{1}, f_{2}$ and $g$. 
For complex singularities, this theorem was proved by A.~N\'{e}methi. 
\end{abstract}

\section{Introduction}
Let $f_{1} : (\Bbb{C}^{n}, \mathbf{0}_{2n}) \rightarrow (\Bbb{C}, \mathbf{0}_{2})$ and 
$f_{2} : (\Bbb{C}^{m}, \mathbf{0}_{2m}) \rightarrow (\Bbb{C}, \mathbf{0}_{2})$ 
be holomorphic functions of independent variables 
$\zz = (z_{1}, \dots, z_{n})$ and $\ww = (w_{1}, \dots, w_{m})$. 
Here $\mathbf{0}_{2N}$ is the origin of $\Bbb{C}^{N}$. 
Join theorem for complex singularities is the following. 
\begin{theorem}[Join theorem]
Let $f$ be a holomorphic function on a neighborhood of 
the origin of $\Bbb{C}^{n+m}$ such that $f(\zz, \ww) = f_{1}(\zz) + f_{2}(\ww)$. 
Then 
the Milnor fiber of $f$ is homotopy equivalent to the join of the Milnor fibers of $f_1$ and $f_2$ and 
the monodromy of $f$ is equal to the join of the monodromies of $f_1$ and $f_2$ up to homotopy. 
\end{theorem}
Join theorem is algebraically proved by M. Sebastiani and R. Thom for isolated singularities~\cite{ST}. 
So Join theorem is often called Thom--Sebastiani theorem. 
M. Oka showed this for weighted homogeneous singularities \cite{O-1}. 
For general complex singularities, this is proved by K.~Sakamoto \cite{S2}. 


Let $g : (\Bbb{R}^{N}, \mathbf{0}_{N}) \rightarrow (\Bbb{R}^{p}, \mathbf{0}_{p})$
be a real analytic germ, 
where $N \geq p \geq 2, \mathbf{0}_{N}$ 
and $\mathbf{0}_{p}$ are the origins of $\Bbb{R}^{N}$ and $\Bbb{R}^{p}$ respectively. 
In general, real analytic singularities does not admit Milnor fibrations. 
To show the existence of the Milnor fibration of $g$, 
we assume that $g$ satisfies the following conditions.  
Set $V(g) = g^{-1}(\mathbf{0}_{p}) \cap B^{N}_{\varepsilon}$, 
where $B^{N}_{\varepsilon} = \{\xx \in \Bbb{R}^{N} \mid \| \xx \| \leq \varepsilon \}$. 
In this paper, $B^{N}_{\varepsilon}$ is used for the disk in the defining euclidean space. 
A real analytic germ $g : (\Bbb{R}^{N}, \mathbf{0}_{N}) \rightarrow (\Bbb{R}^{p}, \mathbf{0}_{p})$ is 
\textit{locally surjective near the origin} if there exists a positive real number $\varepsilon$ 
so that for any $\mathbf{x} \in V(g)$, 
there exists an open neighborhood $W$ of $\mathbf{x}$ so that $\mathbf{0}_{p}$ is an 
interior point of the image $g(W)$. 
We also assume that $V(g)$ has codimension $p$ at the origin. 
Let $\varepsilon$ be a small positive real number 
and $\mathcal{S}$ be a stratification of $V(g)$. 
The map $g$ satisfies \textit{$a_{f}$-condition}
if 
$B^{N}_{\varepsilon} \setminus V(g)$ 
has no critical point and satisfies the following condition: 
For any sequence $p_{\nu} \in B^{N}_{\varepsilon} \setminus V(g)$ 
such that 
\[
T_{p_{\nu}}g^{-1}(g(p_{\nu})) \rightarrow \tau, \ \ p_{\nu} \rightarrow p_{\infty} \in M, 
\]
where $M \in \mathcal{S}$, 
we have $T_{p_{\infty}}M \subset \tau$. 
We say $\varepsilon$ is \textit{an $a_{f}$-stable radius for $g$ with respect to $\mathcal{S}$} 
if it satisfies the following: 
Each sphere $S^{N-1}_{\varepsilon'}, 0 < \varepsilon' \leq \varepsilon$ intersects transversely with 
any stratum of $\mathcal{S}$ 
and $\mathbf{0}_{p}$ is the only critical value of 
$g|_{B^{N}_{\varepsilon}} : B^{N}_{\varepsilon} \rightarrow \Bbb{R}^{p}$. 

Since $V(g)$ is a real analytic set, we may assume that 
$\mathcal{S}$ is a Whitney stratification. 
See \cite{H} for further information. 
We take $\varepsilon$ sufficiently small if necessary. 
Then the sphere $\partial B^{N}_{\varepsilon}$ 
intersects $M$ 
transversely for any $M \in \mathcal{S}$. 
See \cite[Corollary 2.9]{M2} and the proof of \cite[Lemma 3.2]{BV}. 
Take an $a_{f}$-stable radius $\varepsilon$ for $g$ and take a sufficiently small $\delta$, 
$0 < \delta \ll \varepsilon$ so that $g^{-1}(\eta)$ intersects transversely with the sphere of radius $\varepsilon$. 
See \cite[Proposition 11]{O2}. 
By the above conditions and the Ehresmann fibration theorem \cite{W}, we may assume that 
$\delta \ll \varepsilon$ and 
\[
g : B^{N}_{\varepsilon} \cap g^{-1}(D^{p}_{\delta} \setminus \{\mathbf{0}_{p}\}) 
\rightarrow D^{p}_{\delta} \setminus \{\mathbf{0}_{p}\} 
\]
is a locally trivial fibration, 
where 
$D^{p}_{\delta} = \{\ww \in \Bbb{R}^{p} \mid \| \ww \| \leq \delta \}$. 
The isomorphism class of the above fibration does not depend on the choice of $\varepsilon$ and $\delta$. 
We call this fibration 
\textit{stable tubular Milnor fibrations of $g$}. 



Let $f_{1} : (\Bbb{R}^{n}, \mathbf{0}_{n}) \rightarrow (\Bbb{R}^{p}, \mathbf{0}_{p})$ 
and $f_{2} : (\Bbb{R}^{m}, \mathbf{0}_{m}) \rightarrow (\Bbb{R}^{p}, \mathbf{0}_{p})$ 
be analytic germs, 
where $n, m \geq p \geq 2$. 
Set $V(f_{1}) = f_{1}^{-1}(\mathbf{0}_{p}) \cap B^{n}_{\varepsilon}$ and 
$V(f_{2}) = f_{2}^{-1}(\mathbf{0}_{p}) \cap B^{m}_{\varepsilon}$
for $0 < \varepsilon \ll 1$. 
We denote a stratification of 
$V(f_{1})$ (resp. $V(f_{2})$) 
by $\mathcal{S}_{1}$ (resp. $\mathcal{S}_{2}$). 
Assume that $f_{1}$ and $f_{2}$ satisfy the following conditions: 
\renewcommand{\theenumi}{\roman{enumi}}
\begin{enumerate}[({a}-i)]
\item
$f_{j}$ has an isolated value at the origin, 
$\text{codim}_{\Bbb{R}} V(f_{j}) = p$ 
and $f_{j}$ is locally surjective on $V(f_{j})$ near the origin for $j = 1, 2$, 
\item
$f_{j}$ satisfies $a_{f}$-condition with respect to $\mathcal{S}_{j}$ for $j = 1, 2$. 
\end{enumerate}
\noindent
Take a common $a_{f}$-stable radius $\varepsilon$ for $f_{1}$ and $f_{2}$ and take a sufficiently small $\delta$, 
$0 < \delta \ll \varepsilon$ so that 
$f_{j}^{-1}(\eta)$ intersects transversely with the sphere of radius $\varepsilon$ 
for $j = 1, 2$ and $\lvert \eta \rvert \leq \delta$. 
Set $U_{j}(\varepsilon, \delta) = \{ \xx \in B^{n_{j}}_{\varepsilon} \mid \| f_{j}(\xx)\| \leq \delta \}$ 
with $n_{1} = n$ and $n_{2} = m$. 
By the above conditions and the Ehresmann fibration theorem \cite{W}, we may assume that 
\[
f_{j} : U_{j}(\varepsilon, \delta) \setminus V(f_{j}) \rightarrow D^{p}_{\delta} \setminus \{\mathbf{0}_{p}\}
\]
is a locally trivial fibration for $j = 1, 2$. 
Put $V(f_{1} + f_{2}) = (f_{1} + f_{2})^{-1}(0) \cap (U_{1}(\varepsilon, \delta) \times U_{2}(\varepsilon, \delta))$. 
By \cite[Proposition 5.2]{ACT}, $f_{1} + f_{2}$ also satisfies the conditions (a-i) and (a-ii). 
In \cite{In1}, 
the fiber of the tubular Milnor fibration of $f_{1} + f_{2}$ is homotopy equivalent to 
the join of the fibers of the tubular Milnor fibrations of $f_1$ and $f_2$. 
Moreover, if $p = 2$, the monodromy of the tubular Milnor fibration of $f_{1} + f_{2}$ is equal to 
the join of the monodromies of $f_1$ and $f_2$ up to homotopy. 
L.~H.~Kauffman and W. D. Neumann studied fiber structures and Seifert forms of 
links defined by tame isolated singularities of real analytic germs of independent variables \cite{KN}. 
For mixed weighted homogeneous singularities, 
Join theorem is proved by J. L. Cisneros-Molina~\cite{C}. 


A.~N\'{e}methi studied a generalized Join theorem for complex analytic singularities. 
Let $f_{1} : (\Bbb{C}^{n}, \mathbf{0}_{2n}) \rightarrow (\Bbb{C}, \mathbf{0}_{2}), 
f_{2} : (\Bbb{C}^{m}, \mathbf{0}_{2m}) \rightarrow (\Bbb{C}, \mathbf{0}_{2})$ 
and $g : (\Bbb{C}^{2}, \mathbf{0}_{4}) \rightarrow (\Bbb{C}, \mathbf{0}_{2})$ 
be complex analytic germs of independent variables. 
Then a complex analytic germ 
$f : (\Bbb{C}^{n} \times \Bbb{C}^{m}, \mathbf{0}_{2n+2m}) \rightarrow (\Bbb{C}, \mathbf{0}_{2})$ 
is defined by $f(\xx, \yy) = g(f_{1}(\xx), f_{2}(\yy))$. 
Let $F_{1}, F_{2}$ and $F_{g}$ be the Milnor fibers of $f_{1}, f_{2}$ and $g$ respectively. 
We denote the total space of a fiber bundle with base space $F_{g}$ and fiber $F_{1} \times F_{2}$ 
by $F'$. 
Set $\tilde{F}_{1} = V(f_{1}) \times F_{2}$ and $\tilde{F}_{2} = F_{1} \times V(f_{2})$. 
N\'{e}methi showed that 
the Milnor fiber of $f$ has the homotopy type of a space obtained from $F'$ by gluing to 
a fiber $F_{1} \times F_{2}$ $n_{1}$ copies of $\tilde{F}_{1}$ and $n_{2}$ copies of 
$\tilde{F}_{2}$, where 
$n_{1}$ is the number of $\{ (0, z_{2}) \in B^{2m}_{\varepsilon} \cap g^{-1}(\delta) \}$ and 
$n_{2}$ is the number of $\{ (z_{1}, 0) \in B^{2n}_{\varepsilon} \cap g^{-1}(\delta) \}$.


To study a generalization of Join theorem for real analytic singularities, 
we consider strongly non-degenerated mixed functions. 
Let $(g_{1}, g_{2}) : (\Bbb{R}^{2n}, \mathbf{0}_{2n}) \rightarrow (\Bbb{R}^{2}, \mathbf{0}_{2})$ 
be an analytic map germ 
with real $2n$-variables $x_{1}, \dots, x_{n}$ and $y_{1}, \dots, y_{n}$. 
Then $(g_{1}, g_{2})$ is represented by 
a complex-valued function of variables $\zz = (z_{1}, \dots, z_{n})$ and 
$\bar{\zz} = (\bar{z}_{1}, \dots, \bar{z}_{n})$ as 
\[ 
g(\zz, \bar{\zz}) := 
g_{1}\Bigl(\frac{\zz + \bar{\zz}}{2}, \frac{\zz - \bar{\zz}}{2\sqrt{-1}}\Bigr) 
+\sqrt{-1}g_{2}\Bigl(\frac{\zz + \bar{\zz}}{2}, \frac{\zz - \bar{\zz}}{2\sqrt{-1}}\Bigr). 
\]
Here any complex variable $z_{j}$ of $\Bbb{C}^{n}$ is represented by 
$x_{j} + \sqrt{-1}y_{j}$ and $\bar{z}_{j}$ is the complex conjugate of $z_j$ for $j = 1, \dots, n$. 
Then a map $g : (\Bbb{C}^{n}, \mathbf{0}_{2n}) \rightarrow (\Bbb{C}, \mathbf{0}_{2})$ 
is called a \textit{mixed function map}. 
Oka introduced the notion of Newton boundaries of mixed functions and 
the concept of strong non-degeneracy. 
If $g$ is a convenient strongly non-degenerate mixed function or 
a strongly non-degenerate mixed function which 
is locally tame along vanishing coordinate subspaces, 
then $g$ satisfies the conditions (a-i) and (a-ii). 
See \cite{O1, O2}.

Assume that 
$f_{1} : (\Bbb{R}^{n}, \mathbf{0}_{n}) \rightarrow (\Bbb{R}^{2}, \mathbf{0}_{2})$ 
and $f_{2} : (\Bbb{R}^{m}, \mathbf{0}_{m}) \rightarrow (\Bbb{R}^{2}, \mathbf{0}_{2})$ 
satisfy the conditions (a-i) and (a-ii). 
Let $g$ be a $2$-variable strongly non-degenerate mixed polynomial which 
is locally tame along vanishing coordinate subspaces. 
Then a real analytic germ 
$f : (\Bbb{R}^{n} \times \Bbb{R}^{m}, \mathbf{0}_{n+m}) \rightarrow (\Bbb{R}^{2}, \mathbf{0}_{2})$ 
is defined by $f(\xx, \yy) = g(f_{1}(\xx), f_{2}(\yy))$. 
In general, $f$ is not strongly non-degenerate. 
To show the existence of the Milnor fibration of $f$, 
we need to prove that $f$ satisfies $a_{f}$-condition. 
Set $V(f) = f^{-1}(\mathbf{0}_{2}) \cap (U_{1}(\varepsilon, \delta) \times U_{2}(\varepsilon, \delta))$. 
We take the stratification $\mathcal{S}_{f}$ of 
$V(f)$ as follows: 
\[
\mathcal{S}_{f} : (\mathcal{S}_{1} \times \mathcal{S}_{2}) \sqcup \mathcal{S}' 
\sqcup 
\Bigl( V(f) \setminus (\mathcal{S}_{1} \times \mathcal{S}_{2} \sqcup \mathcal{S}')\Bigr), 
\]
where $\mathcal{S}' = \{ \mathcal{S}_{j} \times (U_{(j+1)\bmod 2}(\varepsilon, \delta) 
\setminus V(f_{(j+1)\bmod 2} \mid 
\{j\} \in \mathcal{I}_{v}(g), j = 1, 2 \}$. 
By using $\mathcal{S}_{f}$, we can show the following theorem. 

\begin{theorem}\label{a_f}
Let $f_{1} : (\Bbb{R}^{n}, \mathbf{0}_{n}) \rightarrow (\Bbb{R}^{2}, \mathbf{0}_{2})$ 
and $f_{2} : (\Bbb{R}^{m}, \mathbf{0}_{m}) \rightarrow (\Bbb{R}^{2}, \mathbf{0}_{2})$ 
be real analytic germs of independent variables $\xx = (x_{1}, \dots, x_{n})$ 
and $\yy = (y_{1}, \dots, y_{m})$, where $n, m \geq 2$. 
Assume that $f_{1}$ and $f_{2}$ satisfy the conditions (a-i) and (a-ii). 
Let $g$ be a $2$-variable strongly non-degenerate mixed polynomial which 
is locally tame along vanishing coordinate subspaces. 
Then the real analytic germ 
$f = g \circ (f_{1}, f_{2})$ 
satisfies $a_{f}$-condition with respect to $\mathcal{S}_{f}$. 
\end{theorem}

By Theorem \ref{a_f}, we can show that $f$ admits the Milnor fibration. 
So we can study N\'{e}methi' theorem for $f$. 

\begin{theorem}
Let $b_{g} \subset \Bbb{C}^{2}$ be a bouquet of circles with base point $\ast$. 
Assume that $b_{g}$ is homotopy equivalent to the Milnor fiber of $g$ and 
$b_{g} \cap \{ z_{1}z_{2} = 0\} = \emptyset$. 
Set $\tilde{F}_{1} = V(f_{1}) \times F_{2}$ and $\tilde{F}_{2} = F_{1} \times V(f_{2})$. 
Then the Milnor fiber $F_{f}$ of $f$ is homotopy equivalent to a space obtained from $(f_{1}, f_{2})^{-1}(b_{g})$ 
by gluing to $(f_{1}, f_{2})^{-1}(\ast)$ $n_{1}$ copies of $\tilde{F}_{1}$ and $n_{2}$ copies of 
$\tilde{F}_{2}$, where 
$n_{1}$ is the number of $\{ (0, z_{2}) \in B^{m}_{\varepsilon} \cap g^{-1}(\delta) \}$ and 
$n_{2}$ is the number of $\{ (z_{1}, 0) \in B^{n}_{\varepsilon} \cap g^{-1}(\delta) \}$. 
\end{theorem}

This paper is organized as follows. 
In Section $2$ we give the definition of strongly non-degenerate mixed functions. 
In Section $3$ we prove Theorem~$2$ and the existence of the Milnor fibration of $f$. 
In Section $4$ we prove Theorem~$3$.
In Section $5$ we study the zeta function of the monodromy of $f$. 

\section{Strongly non-degenerate mixed functions}
In this section, we introduce a class of mixed functions which admit 
tubular Milnor fibrations and spherical Milnor fibrations 
given by Oka in \cite{O1}. 
Let $g(\zz, \bar{\zz})$ be a mixed function, i.e., 
$g(\zz, \bar{\zz})$ is a function 
expanded in a convergent power series 
of variables $\zz = (z_1, \dots, z_n)$ and $\bar{\zz} = (\bar{z}_1, \dots, \bar{z}_n)$ 
\[
g(\zz, \bar{\zz}) := \textstyle \sum_{\nu, \mu} c_{\nu, \mu}\zz^{\nu}\bar{\zz}^{\mu}, 
\]
where $\zz^{\nu} = z^{\nu_1}_1 \cdots z^{\nu_n}_n$ for $\nu = (\nu_1, \dots, \nu_n)$ 
(respectively $\bar{\zz}^{\mu} = \bar{z}_{1}^{\mu_1} \cdots \bar{z}_{n}^{\mu_n}$ for $\mu = (\mu_1, \dots, \mu_n))$. 
The \textit{Newton polygon} $\Gamma_{+}(g;\zz.\bar{\zz})$ is defined by the convex hull of 
\[
\textstyle \bigcup_{(\nu, \mu)}\{(\nu + \mu)+\Bbb{R}^n_{+} \ | 
\ c_{\nu, \mu}\neq0\}, 
\]
where $\nu + \mu$ is the sum of the multi-indices of $\zz^{\nu}\bar{\zz}^{\mu}$, 
i.e., $\nu+\mu = (\nu_{1}+\mu_{1}, \dots, \nu_{n}+\mu_{n})$. 
\textit{The Newton boundary} $\Gamma(g;\zz,\bar{\zz})$ is 
the union of compact faces  of $\Gamma_+(g;\zz,\bar{\zz})$. 
Let $\Bbb{Z}_{+}$ be the set of non-negative integers. 
For any weight vector $P = {}^{t}(p_{1}, \dots, p_{n}) \in (\Bbb{Z}_{+})^{n}$, 
we define 
a linear function $\ell_{P}$ on $\Gamma_+(g;\zz,\bar{\zz})$ as follows:
\[
\xi = (\xi_{1}, \dots, \xi_{n}) \mapsto \textstyle \sum_{j=1}^{n} p_{j}\xi_{j}. 
\] 
We denote the minimal value of $\ell_{P}$ by $d(P)$ and 
put $\Delta(P) = \{ \xi \in \Gamma_{+}(g) \mid \ell_{P}(\xi) = d(P) \}$. 
Let $\Delta$ and $P$ be a face of $\Gamma_+(g;\zz,\bar{\zz})$ and a weight vector respectively, 
we define 
\[
g_{\Delta}(\zz,\bar{\zz}) = \textstyle \sum_{(\nu + \mu) \in \Delta}c_{\nu, \mu}\zz^{\nu}\bar{\zz}^{\mu}, \ \ \ 
g_{P} = \textstyle \sum_{\nu + \mu \in \Delta(P)}c_{\nu, \mu}\zz^{\nu}\bar{\zz}^{\mu}. 
\]
The mixed functions $g_{\Delta}$ and $g_{P}$ are called the \textit{face function of $f$ of the face $\Delta$} 
and \textit{face function of $f$ of the weight $P$} respectively.

The \textit{strongly non-degeneracy of mixed functions} is defined from the Newton boundary as follows:  
let $\Delta_1,$ $\dots$ ,$\Delta_m$ be the faces of $\Gamma(g;\zz,\bar{\zz})$. 
If $g_{\Delta_{k}}(\zz,\bar{\zz}) : \Bbb{C}^{\ast n} \rightarrow \Bbb{C}$ 
has no critical point, and 
$g_{\Delta_{k}}$ is surjective if $\dim \Delta_{k} \geq 1$, 
we say that $g(\zz, \bar{\zz})$ is \textit{strongly non-degenerate} for $\Delta_{k}$, where 
$\Bbb{C}^{\ast n} = \{\zz = (z_{1}, \dots, z_{n}) \mid z_{j} \neq 0, j = 1, \dots, n \}$. 
If $g(\zz, \bar{\zz})$ is strongly non-degenerate for any $\Delta_{k}$ for $k = 1,\dots,m$,
we say that $g(\zz, \bar{\zz})$ is 
{\textit{strongly non-degenerate}}. 
If $g( (0,\dots,0, z_{j},0,\dots,0), (0,\dots,0, \bar{z}_{j},0,\dots,0)) \not\equiv 0$ 
for each $j= 1,\dots,n$, 
then we say that $g(\zz, \bar{\zz})$ is \textit{convenient}.

For a subset $I \subset \{1, \dots, n\}$, we set 
\[
\Bbb{C}^{I} = \{ (z_{1}, \dots, z_{n}) \in \Bbb{C}^{n} \mid z_{i} = 0, i \not\in I \}, \ \ \ 
\Bbb{C}^{\ast I}= \{ (z_{1}, \dots, z_{n}) \in \Bbb{C}^{n} \mid z_{i} \neq 0 \Leftrightarrow i \in I \}. 
\] 
Put $g^{I} = g|_{\Bbb{C}^{I}}$. Then we can define the subsets of 
$\{I \mid I \subset \{1, \dots, n\} \}$ as follows: 
\[
\mathcal{I}_{nv}(g) = \{ I \subset \{1, \dots, n\} \mid g^{I} \not\equiv 0 \}, \ \ \
\mathcal{I}_{v}(g) = \{ I \subset \{1, \dots, n\} \mid g^{I} \equiv 0 \}. 
\]
If $I \in \mathcal{I}_{v}(g)$, $\Bbb{C}^{I}$ is called a \textit{vanishing coordinates subspace}. 
For $I \in \mathcal{I}_{v}(g)$, we define the distance function on $\Bbb{C}^{I}$ by 
$\rho_{I}(\zz) = \sqrt{\sum_{i \in I}\lvert z_{i}\rvert^{2}}$. 
We say that $g$ is 
\textit{locally tame along the vanishing coordinates subspace $\Bbb{C}^{I}$} 
if there exists a positive real number $r_{I}$ 
such that for any $\mathbf{a}_{I} = (\alpha_{i})_{i \in I} \in \Bbb{C}^{\ast I}$ with 
$\rho_{I}(\mathbf{a}_{I}) \leq r_{I}$ and for any weight vector $P = {}^{t}(p_{1}, \dots, p_{n})$ 
with 
$I(P) = \{ i \mid p_{i} = 0\} = I, 
g_{P}|_{z_{I} = 
\mathbf{a}_{I}}$ is strongly non-degenerate 
as a function of $\{z_{j} \mid j \in I^{c}\}$. 
A mixed function 
$g$ is called \textit{locally tame} if $g$ is locally tame for any vanishing coordinates subspace. 
If a strongly non-degenerate mixed function $g$ 
is convenient or locally tame for any vanishing coordinates subspace, 
$g$ has both tubular and spherical Milnor fibrations and also two fibrations are isomorphic \cite{O1, O2}.  
Moreover 
$g^{-1}(0) \cap B^{2n}_{\varepsilon}$ has the following stratification. 

\begin{theorem}[\cite{O2}]
Let $g$ be a strongly non-degenerate mixed polynomial. 
Assume that $g$ is 
locally tame for any vanishing coordinates subspace. 
Set 
\[\
\mathcal{S}_{can} := \{ g^{-1}(0) \cap \Bbb{C}^{\ast I}, 
\Bbb{C}^{\ast I} \setminus (g^{-1}(0) \cap \Bbb{C}^{\ast I}) \mid I \in \mathcal{I}_{nv}(g)\} 
\cup \{ \Bbb{C}^{\ast I} \mid I \in \mathcal{I}_{v}(g)\}. 
\]
Then 
$g$ satisfies $a_{f}$-condition with respect to $\mathcal{S}_{can}$ in $B^{2n}_{\varepsilon}$ . 
\end{theorem}


Let $g_t$ be an analytic family of strongly non-degenerate mixed polynomials 
which are locally tame along vanishing coordinate subspaces. 
Assume that the Newton boundary of $g_t$ is constant for $0 \leq t \leq 1$. 
C. Eyral and M. Oka showed that the topological type of $(V(g_{t}), \mathbf{0}_{2n})$ is constant for any $t$ 
and their tubular Milnor fibrations are equivalent \cite{EO}.

\section{The existence of the Milnor fibration of $f$}
Let 
$f_{1} : (\Bbb{R}^{n}, \mathbf{0}_{n}) \rightarrow (\Bbb{R}^{2}, \mathbf{0}_{2})$ 
and $f_{2} : (\Bbb{R}^{m}, \mathbf{0}_{m}) \rightarrow (\Bbb{R}^{2}, \mathbf{0}_{2})$ 
be real analytic germs of independent variables, where $n, m \geq 2$. 
We take a stratification $\mathcal{S}_{j}$ of $f_{j}^{-1}(\mathbf{0}_{2}) \cap B^{n_{j}}_{\varepsilon}$ 
with $n_{1} = n$ and $n_{2} = m$. 
Assume that $f_{j}$ satisfies the conditions (a-i) and (a-ii) with respect to 
$\mathcal{S}_{j}$ for $j = 1, 2$. 
Let $g$ be a $2$-variable strongly non-degenerate mixed polynomial which 
is locally tame along vanishing coordinate subspaces. 
Then a real analytic germ 
$f : (\Bbb{R}^{n} \times \Bbb{R}^{m}, \mathbf{0}_{n+m}) \rightarrow (\Bbb{R}^{2}, \mathbf{0}_{2})$ 
is defined by $f(\xx, \yy) = g(f_{1}(\xx), f_{2}(\yy))$. 
In this section, we prove the existence of the Milnor fibration of $f$. 



\begin{lemma}
The origin $\mathbf{0}_{2}$ is an isolated critical value of $f$. 
\end{lemma}

\begin{proof}
For any $(\xx, \yy) \in (U_{1}(\varepsilon, \delta) \times U_{2}(\varepsilon, \delta)) \setminus V(f)$, 
we show that 
the rank of $Jf((\xx, \yy)) = 2$, 
where $Jf$ is the Jacobian matrix of $f$. 
Set $g_{1} = \Re g, g_{2} = \Im g, z_{j1} = \Re z_{j}$ and $z_{j2} = \Im z_{j}$ for $j = 1, 2$. 
Put 
\[
G_{1} = \begin{pmatrix}
         \frac{\partial g_{1}}{\partial z_{11}} & \frac{\partial g_{1}}{\partial z_{12}} \\
         \frac{\partial g_{2}}{\partial z_{11}} & \frac{\partial g_{2}}{\partial z_{12}} \\
         \end{pmatrix}, \ \ \ 
G_{2} = \begin{pmatrix}
         \frac{\partial g_{1}}{\partial z_{21}} & \frac{\partial g_{1}}{\partial z_{22}} \\
         \frac{\partial g_{2}}{\partial z_{21}} & \frac{\partial g_{2}}{\partial z_{22}} \\
         \end{pmatrix}. 
\]
Since $f = g \circ (f_{1}, f_{2})$, the Jacobian matrix $Jf$ of $f$ is equal to 
\[
\begin{pmatrix}
G_{1} &  G_{2} 
\end{pmatrix}
\begin{pmatrix}
Jf_{1} &  O' \\
O & Jf_{2}
\end{pmatrix}  
= 
\begin{pmatrix}
G_{1}Jf_{1} &  G_{2}Jf_{2} 
\end{pmatrix}, 
\]
where $O$ is the $2 \times n$ zero matrix and $O'$ is the $2 \times m$ zero matrix . 

Suppose that $f_{1}(\xx) = \mathbf{0}_{2}, \{2\} \in \mathcal{I}_{nv}(g)$ 
and $f_{2}(\yy) \neq \mathbf{0}_{2}$. 
Since $g$ is strongly non-degenerate, $g|_{\Bbb{C}^{\ast \{2\}}} : \Bbb{C}^{\ast \{2\}} \rightarrow \Bbb{C}$ 
is a regular function. 
By the condition (a-i), 
$\text{rank}\>Jf_{2}(\yy) = \text{rank}\>G_{2} = 2$. 
Thus $\text{rank}\>Jf(\xx, \yy)$ is equal to $2$. 
If $f_{1}(\xx) \neq \mathbf{0}_{2}, f_{2}(\yy) = \mathbf{0}_{2}$ and  $\{1\} \in \mathcal{I}_{nv}(g)$, 
by using same argument, we can show that $\text{rank}\>Jf(\xx, \yy) = 2$.

Assume that $(\xx, \yy)$ satisfies 
$f_{1}(\xx) \neq \mathbf{0}_{2}, f_{2}(\yy) \neq \mathbf{0}_{2}$ 
and 
$f(\xx, \yy) = g(f_{1}(\xx), f_{2}(\yy)) \neq \mathbf{0}_{2}$. 
Since $f_{1}, f_{2}$ and $g$ have an isolated critical value at the origin, we have 
\[ 
\text{rank}\>Jf_{1}(\xx) = \text{rank}\>Jf_{2}(\yy) = 
\text{rank}\>Jg(f_{1}(\xx), f_{2}(\yy)) = 2. 
\]
If $\text{rank}\>G_{1} = 2$ or $\text{rank}\>G_{2} = 2$, 
the rank of $Jf(\xx, \yy)$ 
is equal to $2$.

Suppose that $\text{rank}\>G_{1} < 2$ and $\text{rank}\>G_{2} < 2$. 
Set 
\[
\Bigl(\frac{\partial g_{1}}{\partial z_{11}}, \frac{\partial g_{1}}{\partial z_{12}}, 
\frac{\partial g_{1}}{\partial z_{21}}, \frac{\partial g_{1}}{\partial z_{22}}\Bigr) = 
(a_{1}, a_{2}, b_{1}, b_{2}).
\]
Since $\text{rank}\>Jg(f_{1}(\xx), f_{2}(\yy)) = 2$, 
we may assume that 
$(a_{1}, a_{2}, b_{1}, b_{2}) \neq (0, 0, 0, 0)$. 
If $(a_{1}, a_{2}) \neq (0, 0)$ and $(b_{1}, b_{2}) \neq (0, 0)$, 
then there exist real numbers $r$ and $s$ such that 
$Jg(f_{1}(\xx), f_{2}(\yy)) = (G_{1}\ G_{2})$ is equal to 
\[
\begin{pmatrix}
a_{1} & a_{2} & b_{1} & b_{2} \\
ra_{1} & ra_{2} & sb_{1} & sb_{2} 
\end{pmatrix}.
\]
Since $\text{rank}\>Jg(f_{1}(\xx), f_{2}(\yy)) = 2$, $r \neq s$. 
Set $f_{j} = (f_{j1}, f_{j2})$, where 
$f_{j1}$ and $f_{j2}$ are real-valued functions for $j = 1, 2$. 
Then $Jf((\xx, \yy))$ is equal to 
\[
\begin{pmatrix}
a_{1}d f_{11} + a_{2}d f_{12} & b_{1}d f_{21} +b_{2}d f_{22} \\
r(a_{1}d f_{11} + a_{2}d f_{12}) & s(b_{1}d f_{21} +b_{2}d f_{22})
\end{pmatrix}. 
\]
Here $d f_{jk}$ is the gradient of a smooth function of $f_{jk}$ for 
$j = 1, 2$ and $k = 1, 2$. 
Since $\text{rank}\>Jf_{1}(\xx) = \text{rank}\>Jf_{2}(\yy) = 2$, 
we have 
\[
a_{1}d f_{11} + a_{2}d f_{12} \neq O, \ \ \  
b_{1}d f_{21} +b_{2}d f_{22} \neq O'. 
\]
Note that $r$ is not equal to $s$. 
Thus the rank of $Jf((\xx, \yy))$ is equal to $2$. 

If $(a_{1}, a_{2}) \neq (0, 0)$ and $(b_{1}, b_{2}) = (0, 0)$, 
then $Jg(f_{1}(\xx), f_{2}(\yy)) = (G_{1}\ G_{2})$ is equal to 
\[
\begin{pmatrix}
a_{1} & a_{2} & 0 & 0 \\
ra_{1} & ra_{2} & b'_{1} & b'_{2} 
\end{pmatrix}.
\]
Hence $Jf((\xx, \yy))$ is equal to 
\[
\begin{pmatrix}
a_{1}d f_{11} + a_{2}d f_{12} & O' \\
r(a_{1}d f_{11} + a_{2}d f_{12}) & b'_{1}d f_{21} +b'_{2}d f_{22}
\end{pmatrix}. 
\]
Since $\text{rank}\>Jf_{1}(\xx) = \text{rank}\>Jf_{2}(\yy) = 2$, 
$a_{1}d f_{11} + a_{2}d f_{12} \neq O$ and 
$b'_{1}d f_{21} +b'_{2}d f_{22} \neq O'$. 
Thus
$\text{rank}\>Jf((\xx, \yy))$ is equal to $2$. 
If $(a_{1}, a_{2}) = (0, 0)$ and $(b_{1}, b_{2}) \neq (0, 0)$, 
by using the same argument, 
we can show that $\text{rank}\>Jf((\xx, \yy)) = 2$. 
\end{proof}

\begin{proof}[Proof of Theorem \ref{a_f}]
By Curve Selection Lemma, 
there exists a real analytic curve $(\xx(t), \yy(t)) \in U_{1}(\varepsilon, \delta) \times U_{2}(\varepsilon, \delta)$ 
such that 
\begin{itemize}
\item
$(\xx(0), \yy(0)) \in f^{-1}(\mathbf{0}_{2})$, 
\item
$f_{1}(\xx(t)) \not\equiv \mathbf{0}_{2}$ and  
$f_{2}(\yy(t)) \not\equiv \mathbf{0}_{2}$ for $t \neq 0$, 
\item
$f(\xx(t), \yy(t)) \neq \mathbf{0}_{2}$ for $t \neq 0$. 
\end{itemize}

Suppose that $f_{1}(\xx(0)) \neq \mathbf{0}_{2}$ and $f_{2}(\yy(0)) \neq \mathbf{0}_{2}$. 
Since $f_{1}$ and $f_{2}$ satisfy the condition (a-i) and $g$ is strongly non-degenerate, 
$(\xx(0), \yy(0))$ is a regular point of $f$. 

Take $M_{1} \in \mathcal{S}_{1}$ and $M_{2} \in \mathcal{S}_{2}$. 
Suppose that $\xx(0) \in M_{1}$ and $\yy(0) \in M_{2}$. 
Since $f_{1}$ and $f_{2}$ satisfy the condition (a-ii), we have 
\begin{equation*}
\begin{split}
&\lim_{t \rightarrow 0} T_{(x(t), y(t))}f^{-1}\bigl(f(\xx(t), \yy(t))\bigr) \\
\supset &\lim_{t \rightarrow 0} \Bigl(T_{\xx(t)}f_{1}^{-1}\bigl(f_{1}(\xx(t))\bigr) \times 
T_{\yy(t)}f_{2}^{-1}\bigl(f_{2}(\yy(t))\bigr)\Bigr)  \\
\supset &\ T_{(\xx(0), \yy(0))} M_{1} \times M_{2}. 
\end{split}
\end{equation*} 

Suppose that $f_{1}(\xx(0)) = \mathbf{0}_{2}$ and 
$\{ 2\} \in \mathcal{I}_{v}(g)$. 
By Theorem $4$, 
there exist vectors 
\begin{equation*}
\begin{split}
\mathbf{v}_{g,1}(t) &= (g_{1,1}, g_{1,2}, 0, 0)t^{r} + (\text{higher terms}), \\
\mathbf{v}_{g,2}(t) &= (g_{2,1}, g_{2,2}, 0, 0)t^{r} + (\text{higher terms})
\end{split}
\end{equation*}
such that 
$\text{rank}\>\begin{pmatrix}
g_{1,1} &  g_{1,2} \\
g_{2,1} & g_{2,2}
\end{pmatrix} = 2$ 
and 
$\lim_{t \rightarrow 0} T_{(f_{1}(\xx(t)), f_{2}(\yy(t)))}
g^{-1}\bigl(g(f_{1}(\xx(t)), f_{2}(\yy(t)))\bigr)$ is orthogonal to 
$(g_{1,1}, g_{1,2}, 0, 0)$ and $(g_{2,1}, g_{2,2}, 0, 0)$. 

Since $f_{1}$ satisfies $a_{f}$-condition with respect to $\mathcal{S}_{1}$, 
there exist vectors 
\[
\mathbf{v}_{f_{1},1}(t) = \mathbf{a}_{1}t^{s} + (\text{higher terms}), \ 
\mathbf{v}_{f_{1},2}(t) = \mathbf{a}_{2}t^{s} + (\text{higher terms})
\]
such that 
\[
\lim_{t \rightarrow 0} T_{\xx(t)}
f_{1}^{-1}\bigl(f_{1}(\xx(t))\bigr) = \mathbf{a}_{1}^{\bot} \cap \mathbf{a}_{2}^{\bot} 
\supset T_{\xx(0)}M_{1}, 
\]
where $\mathbf{a}_{j}^{\bot} = \{ \mathbf{v} \in \Bbb{R}^{n} \mid (\mathbf{v}, \mathbf{a}_{j}) = 0 \}$
for $j = 1, 2$. 
Up to scalar multiplications, 
we may assume that 
$\mathbf{v}_{f_{1},1}(t)$ and 
$\mathbf{v}_{f_{1},2}(t)$ are equal to $d f_{11}$ and $d f_{12}$ respectively. 
Note that $\mathbf{v}_{g,1}(t)$ and $\mathbf{v}_{g,2}(t)$ are linear combinations of 
$d g_{1}$ and $d g_{2}$ for $t \neq 0$. 
See \cite{O2}. 
Since $f$ is the composition of $g$ and $(f_{1}, f_{2})$, 
$T_{(\xx(t), \yy(t))}f^{-1}\bigl(f(\xx(t), \yy(t))\bigr)$ 
is orthogonal to 
\[
\begin{pmatrix}
\mathbf{v}_{g,1}(t) \\
\mathbf{v}_{g,2}(t)
\end{pmatrix}
\begin{pmatrix}
\mathbf{v}_{f_{1},1}(t) & 0 \cdots 0 \\
\mathbf{v}_{f_{1},2}(t) & 0 \cdots 0 \\
0 \cdots 0 & d f_{21}(\yy(t)) \\
0 \cdots 0 & d f_{22}(\yy(t)) \\
\end{pmatrix} = 
\begin{pmatrix}
(g_{1,1}\mathbf{a}_{1} + g_{1,2}\mathbf{a}_{2}, 0, \dots, 0)t^{r+s} + (\text{higher terms}) \\
(g_{2,1}\mathbf{a}_{1} + g_{2,2}\mathbf{a}_{2}, 0, \dots, 0)t^{r+s} + (\text{higher terms}) 
\end{pmatrix}.  
\]
Thus  
$\lim_{t \rightarrow 0} T_{(\xx(t), \yy(t))}f^{-1}\bigl(f(\xx(t), \yy(t))\bigr)$ 
is orthogonal to the following vectors: 
\[
(g_{1,1}\mathbf{a}_{1} + g_{1,2}\mathbf{a}_{2}, 0, \dots, 0), \ \ \
(g_{2,1}\mathbf{a}_{1} + g_{2,2}\mathbf{a}_{2}, 0, \dots, 0). 
\]
Since $\text{rank}\>\begin{pmatrix}
g_{1,1} &  g_{1,2} \\
g_{2,1} & g_{2,2}
\end{pmatrix} = 2$, 
$\mathbf{v} \in \Bbb{R}^{n}$ is orthogonal to $\mathbf{a}_{1}$ and $\mathbf{a}_{2}$ 
if and only if 
$\mathbf{v}$ is orthogonal to $g_{1,1}\mathbf{a}_{1} + g_{1,2}\mathbf{a}_{2}$ and 
$g_{2,1}\mathbf{a}_{1} + g_{2,2}\mathbf{a}_{2}$. 
Thus we have the following inclusion relation: 
\begin{equation*}
\begin{split}
&\lim_{t \rightarrow 0} T_{(\xx(t), \yy(t))}f^{-1}\bigl(f(\xx(t), \yy(t))\bigr) \\
\supset 
     &\lim_{t \rightarrow 0} T_{\xx(t)}f_{1}^{-1}\bigl(f_{1}(\xx(t)\bigr) \times 
     T_{\yy(t)}(U_{2}(\varepsilon, \delta) \setminus V(f_{2})) \\
     \supset &\ T_{\xx(0)}M_{1} \times T_{\yy(0)}(U_{2}(\varepsilon, \delta)  \setminus V(f_{2})). 
\end{split}
\end{equation*}
If $f_{2}(\xx(0)) = \mathbf{0}_{2}$ and 
$\{ 1\} \in \mathcal{I}_{v}(g)$,  
we can use the same argument. 

If $f_{1}(\xx(t)) = \mathbf{0}_{2}$ or $f_{2}(\yy(t)) = \mathbf{0}_{2}$ for $t \geq 0$, 
we can use Whitney regularity of $V(f_{j})$ for $j = 1, 2$. 
Since $f_{1}$ and $f_{2}$ satisfy the condition (a-ii) and $g$ is strongly non-degenerate, 
$f$ satisfies $a_{f}$-condition with respect to $\mathcal{S}_{f}$. 
\end{proof}

\begin{example}
Consider $g(\zz, \bar{\zz}) = (d_{1} + c_{1}i)z_{1}\lvert z_{2}\rvert^{2}$. 
Put $\ww = (c_{1} + d_{1}i, z_{2}) \in \Bbb{C}^{2}$, where $c_{1} + d_{1}i \neq 0$. 
Then the normalized gradient of $\Re g$ is given by 
\[
\frac{1}{\sqrt{c_{1}^{2} + d_{1}^{2}}}(d_{1}, -c_{1}, 0, 0). 
\]
When $z_{2} \rightarrow 0$,
we have 
\[
\textstyle \lim_{z_{2} \rightarrow 0} T_{\ww}g^{-1}(g(\ww)) 
\subset \textstyle \lim_{z_{2} \rightarrow 0} T_{\ww}(\Re g)^{-1}(\Re g(\ww))
\not\supset \Bbb{C} \times \{0\}. 
\]
Hence $g$ does not satisfy $a_{f}$-condition with respect to $\mathcal{S}_{can}$ \cite{O2}. 
Let $f_{1} : (\Bbb{R}^{n}, \mathbf{0}_{n}) \rightarrow (\Bbb{R}^{2}, \mathbf{0}_{2})$ 
and $f_{2} : (\Bbb{R}^{m}, \mathbf{0}_{m}) \rightarrow (\Bbb{R}^{2}, \mathbf{0}_{2})$ 
be real analytic germs of independent variables, where $n, m \geq 2$. 
Assume that $f_{1}$ and $f_{2}$ satisfy the conditions (a-i) and (a-ii). 
In this case, $G_{1}(\ww)Jf_{1}(\xx) \neq O$, where 
$\xx \in f_{1}^{-1}(c_{1} + d_{1}i)$. Set $\yy \in f_{2}^{-1}(z_{2})$. 
Then we have 
\[
\textstyle \lim_{z_{2} \rightarrow 0} T_{(\xx, \yy)}f^{-1}(f(\xx, \yy)) \not\supset 
T_{\xx}(U_{1}(\varepsilon, \delta) \setminus V(f_{1})) \times T_{\yy} \mathcal{S}_{2}. 
\]
Thus $f = g \circ (f_{1}, f_{2})$ does not satisfy $a_{f}$-condition with respect to $\mathcal{S}_{f}$. 
We next consider 
$g_{a}(\zz, \bar{\zz}) = z_{1}z_{2}^{a}\bar{z}_{2}$ for $a \geq 2$. 
Then $g_{a}$ is a strongly non-degenerate mixed polynomial which 
is locally tame along vanishing coordinate subspaces. 
By Theorem \ref{a_f}, $f_{a} = g_{a}\circ(f_{1}, f_{2})$ 
satisfies $a_{f}$-condition with respect to $\mathcal{S}_{f_{a}}$. 
\end{example}

\begin{lemma}
The real analytic map 
$f$ is locally surjective on $V(f)$ near the origin and 
the codimension of $V(f)$ is equal to $2$.
\end{lemma}

\begin{proof}
Since $f_{1}$ and $f_{2}$ satisfy the condition (a-i) and $g$ admits the Milnor fibration, 
$f$ is locally surjective on $\mathcal{S}_{1} \times \mathcal{S}_{2}$. 
Let $(\xx, \yy)$ be a point of $\mathcal{S}'$. 
By using the condition (a-i) and the Milnor fibration of $f_{1}, f_{2}$ and $g$, 
we can show that the existence of a neighborhood $W_{(\xx, \yy)}$ of $(\xx, \yy)$ such that 
$\mathbf{0}_{2}$ is an interior point of $f(W_{(\xx, \yy)})$. 

Since $V(f) \setminus (\mathcal{S}_{1} \times \mathcal{S}_{2} \sqcup \mathcal{S}')$ 
is the set of regular points of $f$, 
$f$ is locally surjective on $V(f) \setminus (\mathcal{S}_{1} \times \mathcal{S}_{2} \sqcup \mathcal{S}')$ and 
the codimension of $V(f)$ is equal to~$2$.
\end{proof}

By Theorem $2$, Lemma $2$ and 
the Ehresmann fibration theorem \cite{W}, 
we can show the following corollary. 

\begin{corollary}
There exist a positive real number $\varepsilon_{0}$ such that for any $0 < \varepsilon \leq \varepsilon_{0}$, 
there exists a positive real nymber $\delta(\varepsilon)$ so that 
\[
f|_{B^{n+m}_{\varepsilon} \cap f^{-1}(D^{2}_{\delta} \setminus \{ \mathbf{0}_{2}\})} 
: B^{n+m}_{\varepsilon} \cap f^{-1}(D^{2}_{\delta} \setminus \{ \mathbf{0}_{2}\})
\rightarrow D^{2}_{\delta} \setminus \{ \mathbf{0}_{2}\}
\]
is a locally trivial fibration for $0 < \delta \leq \delta(\varepsilon)$. 
The isomorphism class of this fibration does not depend on the choice of $\varepsilon$ and $\delta$. 
\end{corollary}

\section{Proof of Theorem $3$}


Let $b_{g} \subset \Bbb{C}^{2}$ be a bouquet of circles with base point $\ast$. 
Assume that $b_{g}$ is homotopy equivalent to the Milnor fiber of $g$ and 
$b_{g} \cap \{ z_{1}z_{2} = 0\} = \emptyset$. 
Let $f_{1} : (\Bbb{R}^{n}, \mathbf{0}_{n}) \rightarrow (\Bbb{R}^{2}, \mathbf{0}_{2})$ 
and $f_{2} : (\Bbb{R}^{m}, \mathbf{0}_{m}) \rightarrow (\Bbb{R}^{2}, \mathbf{0}_{2})$ 
be real analytic germs of independent variables which satisfy the conditions (a-i) and (a-ii). 
Set $\tilde{F}_{1} = V(f_{1}) \times F_{2}$ and $\tilde{F}_{2} = F_{1} \times V(f_{2})$. 

Let $\varepsilon$ be a common $a_{f}$-stable radius for $f_{1}$ and $f_{2}$. 
Since $f_{j}^{-1}(\eta)$ intersects transversely with $\partial B^{n_{j}}_{\varepsilon}$ 
for $j = 1, 2$ and $\lvert \eta \rvert \ll \varepsilon$, 
there exists a vector field $\mathbf{v}_{j}$ on $B^{n_{j}}_{\varepsilon}$ 
such that 
\begin{itemize}
\item
$(\mathbf{v}_{1}(\xx), \xx) < 0$ and $(\mathbf{v}_{2}(\yy), \yy) < 0$,
\item
$\mathbf{v}_{1}(\xx)$ is tangent to $f_{1}^{-1}(f_{j}(\xx))$ and 
$\mathbf{v}_{2}(\yy)$ is tangent to $f_{2}^{-1}(f_{j}(\yy))$, 
\end{itemize}
where $0 < \lvert \eta \rvert \leq \delta$ and $j = 1, 2$. 
By using $\mathbf{v}_{1}$ and $\mathbf{v}_{2}$, 
we can choose $0 < \varepsilon < \varepsilon_{1} \ll 1$ and 
$0 < \delta < \delta_{1} \ll \varepsilon_{1}$ such that 
\begin{enumerate}
\item
$f_{j}(B^{n_{j}}_{\varepsilon}) \subset D^{2}_{\delta_{1}}$ and 
\[
f_{j} : B^{n_{j}}_{\varepsilon} \cap f_{j}^{-1}(D^{2}_{\eta} \setminus \{ \mathbf{0}_{2}\}) 
\rightarrow D^{2}_{\eta} \setminus \{ \mathbf{0}_{2}\}
\]
is a locally trivial fibration for $j= 1, 2$ and $\eta = \delta, \delta_{1}$, 
\item
$\bigl(D^{2}_{\delta} \times D^{2}_{\delta}, g^{-1}(\tilde{\delta}) \cap 
(D^{2}_{\delta} \times D^{2}_{\delta})\bigr)$ is homeomorphic to 
(Milnor ball, $g^{-1}(\tilde{\delta})$),
\item
there exists a deformation retract 
\[
r_{j} : B^{n_{j}}_{\varepsilon_{1}} \setminus \text{Int}\>B^{n_{j}}_{\varepsilon} \cap f_{j}^{-1}(D^{2}_{\delta}) 
\rightarrow \partial B^{n_{j}}_{\varepsilon} \cap f_{j}^{-1}(D^{2}_{\delta}) 
\] 
such that 
$r_{j}|_{\partial B^{n_{j}}_{\varepsilon} \cap f_{j}^{-1}(D^{2}_{\delta})} = \text{id}$ 
and 
$f_{j} \circ r_{j} = f_{j}$. 
\end{enumerate}

We take $\tilde{\delta}$ sufficiently small so that 
\begin{itemize}
\item
$g^{-1}(\tilde{\delta})$ is a Milnor fiber in 
$D^{2}_{\delta_{1}} \times D^{2}_{\delta_{1}}$ and 
$D^{2}_{\delta} \times D^{2}_{\delta}$, 
\item
$g^{-1}(\tilde{\delta}) \cap \{z_{1}z_{2} = 0\} \subset D^{2}_{\delta} \times D^{2}_{\delta}$. 
\end{itemize}

\begin{lemma}
Set $F_{\varepsilon, \tilde{\delta}} = f^{-1}(\tilde{\delta}) \cap (B^{n}_{\varepsilon} \times B^{m}_{\varepsilon})$. 
Then 
$(f_{1}, f_{2})^{-1}(D^{2}_{\delta} \times D^{2}_{\delta}) \cap F_{\varepsilon, \tilde{\delta}}$ 
is homotopy equivalent to $F_{\varepsilon, \tilde{\delta}}$. 
\end{lemma}

\begin{proof}
Set $g^{-1}(\tilde{\delta})^{\circ} = g^{-1}(\tilde{\delta}) \cap 
\bigl((D^{2}_{\delta_{1}} \times D^{2}_{\delta_{1}}) \setminus 
\text{Int} (D^{2}_{\delta} \times D^{2}_{\delta})\bigr)$. 
Since the Milnor fibers of $g$ are transversal to small spheres, 
there exists a deformation retract 
$D_{t} : g^{-1}(\tilde{\delta})^{\circ} \rightarrow g^{-1}(\tilde{\delta})^{\circ}$ 
such that ${D}_{0} = \text{id}$ and 
$\text{Im}\>{D}_{1} \in g^{-1}(\tilde{\delta}) \cap \partial \bigl(D^{2}_{\delta} \times D^{2}_{\delta})$. 
By the local triviality of $(f_{1}, f_{2})$, 
there exists the deformation retract 
\begin{equation*}
\begin{split}
\tilde{D}_{t} : (B^{n}_{\varepsilon_{1}} \times B^{m}_{\varepsilon_{1}}) \cap 
(f_{1}, f_{2})^{-1}(g^{-1}(\tilde{\delta})^{\circ}) 
\rightarrow 
(B^{n}_{\varepsilon_{1}} \times B^{m}_{\varepsilon_{1}}) \cap 
(f_{1}, f_{2})^{-1}(g^{-1}(\tilde{\delta})^{\circ}) 
\end{split}
\end{equation*}
so that $\tilde{D}_{t}$ is the lifting of $D_{t}$, 
$\tilde{D}_{0} = \text{id}$ and  
$\text{Im}\>\tilde{D}_{1} \in (f_{1}, f_{2})^{-1}(g^{-1}(\tilde{\delta}) 
\cap \partial (D^{2}_{\delta} \times D^{2}_{\delta}))$.

Define 
$\tilde{r}_{j} : B^{n_{j}}_{\varepsilon_{1}} \cap f_{j}^{-1}(D^{2}_{\delta}) \rightarrow 
B^{n_{j}}_{\varepsilon_{1}} \cap f_{j}^{-1}(D^{2}_{\delta})$ by 
\[
\tilde{r}_{j} = \begin{cases}
                r_{j} & \lvert \xx \rvert \geq \varepsilon \\
                \text{id} & \lvert \xx \rvert \leq \varepsilon
                \end{cases}. 
\]
Then the composed map $(\tilde{r}_{1} \times \tilde{r}_{2}) \circ \tilde{D}_{t}$ 
defines a deformation retract of 
$(f_{1}, f_{2})^{-1}(D^{2}_{\delta} \times D^{2}_{\delta}) \cap F_{\varepsilon, \tilde{\delta}}$ 
in $F_{\varepsilon, \tilde{\delta}}$. 
\end{proof}

We take $0 < \varepsilon' < \varepsilon$ and $0 < \delta' < \delta$. 
Assume that $(\varepsilon', \delta')$ has the same properties of $(\varepsilon, \delta)$ and 
$\tilde{\delta}$ is sufficiently small. 
By using the above argument, we can show that 
the inclusion $(f_{1}, f_{2})^{-1}(D^{2}_{\delta'} \times D^{2}_{\delta'}) \cap F_{\varepsilon', \tilde{\delta}} 
\subset (f_{1}, f_{2})^{-1}(D^{2}_{\delta} \times D^{2}_{\delta}) \cap F_{\varepsilon, \tilde{\delta}}$ 
is a deformation retract. 
So we can show the following corollary.

\begin{corollary}
The inclusion 
$F_{\varepsilon', \tilde{\delta}} \subset F_{\varepsilon, \tilde{\delta}}$ is a homotopy equivalence.
\end{corollary}

By Corollary $2$, we have

\begin{lemma}
Let $F_{f}$ be the Milnor fiber of $f$. 
Then $F_{\varepsilon, \tilde{\delta}}$ has the same homotopy type of $F_{f}$. 
\end{lemma}

\begin{proof}
We choose sufficiently small positive real numbers $\varepsilon_{1} > \varepsilon_{2} > \varepsilon_{3}$. 
Set $F_{\varepsilon_{k}, \tilde{\delta}} = f^{-1}(\tilde{\delta}) \cap 
(B^{n}_{\varepsilon_{k}} \times B^{m}_{\varepsilon_{k}})$ for $k = 1, 2, 3$. 
By Corollary $2$, 
The inclusion 
$F_{\varepsilon_{k+1}, \tilde{\delta}} \subset F_{\varepsilon_{k}, \tilde{\delta}}$ is a homotopy equivalence
for $k =1, 2$. 
Since the fiber $f^{-1}(\tilde{\delta})$ intersects transversely with 
$B^{n+m}_{\varepsilon_{1}}, B^{n+m}_{\varepsilon_{2}}$ and $B^{n+m}_{\varepsilon_{3}}$, 
the inclusion 
$f^{-1}(\tilde{\delta}) \cap B^{n+m}_{\varepsilon_{k+1}} \subset 
f^{-1}(\tilde{\delta}) \cap B^{n+m}_{\varepsilon_{k}}$ is also a homotopy equivalence
for $k =1, 2$. 
Thus the sequence 
\begin{equation*}
\begin{split}
F_{\varepsilon_{1}, \tilde{\delta}} &\supset f^{-1}(\tilde{\delta}) \cap B^{n+m}_{\varepsilon_{1}} 
\supset 
F_{\varepsilon_{2}, \tilde{\delta}} 
\supset f^{-1}(\tilde{\delta}) \cap B^{n+m}_{\varepsilon_{2}} \supset 
F_{\varepsilon_{3}, \tilde{\delta}} \\
\end{split}
\end{equation*}
defines a homotopy equivalence 
$F_{\varepsilon, \tilde{\delta}} \rightarrow F_{f}$. 
See \cite[Proposition 1.1]{Ma} and \cite[Lemma 7]{In1}. 
\end{proof}

\begin{proof}[Proof of Theorem $3$]
Consider the following map 
\[
(f_{1}, f_{2}) : (f_{1}, f_{2})^{-1}(D^{2}_{\delta} \times D^{2}_{\delta}) \cap F_{\varepsilon, \tilde{\delta}}
\rightarrow (D^{2}_{\delta} \times D^{2}_{\delta}) \cap g^{-1}(\tilde{\delta}). 
\]
This map is locally trivial over $g^{-1}(\tilde{\delta}) \setminus \{z_{1}z_{2} = 0\}$ with 
fiber $F_{1} \times F_{2}$.

Let $D_{j}^{2}$ be a small neighborhood of a point of $g^{-1}(\tilde{\delta}) \cap \{z_{1}z_{2} = 0\}$ 
and $\gamma_{j}$ be a path from $b_{g}$ to $D_{j}^{2}$ for $j = 1, \dots, n_{1} + n_{2}$. 
Assume that 
\[
b_{g} \cap \gamma_{j} = \{\ast \},\ \  D_{j}^{2} \cap \gamma_{j} = \{ \text{a point} \} \subset \partial D_{j}^{2} 
\]
and 
\[
D_{j}^{2} \cap D_{j'}^{2} = D_{j}^{2} \cap \gamma_{j'} = \emptyset, \ \ 
\gamma_{j} \cap \gamma_{j'} = \{\ast \} 
\]
for $j = 1, \dots, n_{1} + n_{2}$ and $j \neq j'$. 
Since $(f_{1}, f_{2})$ is locally trivial over $g^{-1}(\tilde{\delta}) \setminus \{z_{1}z_{2} = 0\}$, 
by homotopy lifting property, 
$F_{\varepsilon, \tilde{\delta}}$ is homotopy equivalent to 
\[
(f_{1}, f_{2})^{-1}\Bigl(b_{g} \cup \bigl(\textstyle \bigcup_{j=1}^{n_{1} + n_{2}}D_{j}^{2}\bigr) \cup 
\bigl(\textstyle \bigcup_{j=1}^{n_{1} + n_{2}} \gamma_{j}\bigr) \Bigr). 
\]
See \cite[p. 55]{Si}. 
Let $(z_{1}, z_{2})$ be a point of $g^{-1}(\tilde{\delta}) \cap \{z_{1}z_{2} = 0\}$. 
Then we have 
\[
(f_{1}, f_{2})^{-1}(z_{1}, z_{2}) \cong 
                \begin{cases}
                \tilde{F}_{1} & (z_{1} = 0) \\
                \tilde{F}_{2} & (z_{2} = 0)
                \end{cases}.
\]
Thus $F_{f}$ has the homotopy type of a space obtained from 
$(f_{1}, f_{2})^{-1}(b_{g})$ by gluing to the fiber 
$(f_{1}, f_{2})^{-1}(\ast)$
$n_{1}$ copies of $\tilde{F}_{1}$ and $n_{2}$ copies of $\tilde{F}_{2}$. 
\end{proof}

\begin{corollary}
Let $f : (\Bbb{R}^{n} \times \Bbb{R}^{m}, \mathbf{0}_{n+m}) \rightarrow (\Bbb{R}^{2}, \mathbf{0}_{2})$ 
be the real analytic germ in Theorem $3$. 
Then the Euler characteristic of the Milnor fiber $F_{f}$ of $f$ is given by 
\[
\chi(F_{f}) = \chi(F_{g} \setminus \{z_{1}z_{2} = 0\})\chi(F_{1})\chi(F_{2}) + n_{1}\chi(F_{2}) + n_{2}\chi(F_{1}). 
\]
\end{corollary}

\begin{remark}
Let $g$ be a $2$-variable strongly non-degenerate mixed polynomial which 
is locally tame along vanishing coordinate subspaces. 
Assume that $g' := g|_{z_{1} = 0} = \textstyle \sum_{\nu, \mu}c_{\nu, \mu} z_{2}^{\nu}\bar{z}_{2}^{\mu} \not\equiv 0$. 
Put $g'_{\ell} = \textstyle \sum_{\nu + \mu = \ell}c_{\nu, \mu} z_{2}^{\nu}\bar{z}_{2}^{\mu}$. Then we can write 
\begin{equation*}
\begin{split}
g' &= g'_{\underline{d}} + \cdots + g'_{\overline{d}}, \\
g'_{\underline{d}} &= cz_{2}^{a}\bar{z}_{2}^{b}\textstyle \prod_{j = 1}^{s}(z_{2} + \delta_{k}\bar{z}_{2})^{\mu_{k}}
\end{split}
\end{equation*}
where $\underline{d} = \min \{ \nu + \mu \mid c_{\nu, \mu} \neq 0 \}$ and 
$\overline{d} = \max \{ \nu + \mu \mid c_{\nu, \mu} \neq 0 \}$. 
Suppose that all zero points of $g'$ are regular points of $g'$ and 
$\lvert \delta_{k} \rvert > 1$ for $k = 1, \dots s$. 
By \cite[Theorem 20]{O4}, the number of $g'^{-1}(0) \cap D^{2}_{\delta}$ is equal to 
$a- b + \sum_{k = 1}^{s}\mu_{k}$. 

\end{remark}

\subsection{Spherical Milnor fibrations}
Let $P$ be a real analytic germ which satisfies the conditions (a-i) and (a-ii). 
We assume that $P$ satisfies 
the following condition: 
\begin{enumerate}[({a}-i)]
\setcounter{enumi}{2}
\item
there exists a positive real number $r'$ such that 
\[
P/\lvert P\rvert : \partial B^{N}_{r} \setminus K_{P} \rightarrow S^{p-1}
\]
is a locally trivial fibration and this fibration is isomorphic to the tubular Milnor fibration of $P$, 
where $K_P = \partial B^{N}_{r} \cap P^{-1}(0)$ and $0 < r \leq r'$ 
\end{enumerate}
The fibration in (a-iii) is called \textit{the spherical Milnor fibration of $P$}.

\begin{corollary}
Let $f_{1} : (\Bbb{R}^{n}, \mathbf{0}_{n}) \rightarrow (\Bbb{R}^{p}, \mathbf{0}_{p})$ 
and $f_{2} : (\Bbb{R}^{m}, \mathbf{0}_{m}) \rightarrow (\Bbb{R}^{p}, \mathbf{0}_{p})$ 
be analytic germs in Theorem $3$. 
Assume that $f_{1}, f_{2}$ and 
$f = g(f_{1}, f_{2})$ satisfy the condition (a-iii). 
Set $\overline{F}_{j}$ is the fiber of the spherical Milnor fibration of $f_{j}$ for $j = 1, 2$. 
Then the fiber of the spherical Milnor fibration of $f$ is 
homotopy equivalent to 
a space obtained from $(f_{1}, f_{2})^{-1}(b_{g})$ 
by gluing to $(f_{1}, f_{2})^{-1}(\ast)$ $n_{1}$~copies of $\overline{F}_{1}$ and $n_{2}$ copies of 
$\overline{F}_{2}$, where 
$n_{1}$ is the number of $\{ (0, z_{2}) \in B^{m}_{\varepsilon} \cap g^{-1}(\delta) \}$ and 
$n_{2}$ is the number of $\{ (z_{1}, 0) \in B^{n}_{\varepsilon} \cap g^{-1}(\delta) \}$. 
\end{corollary}

\begin{proof}
By Theorem $3$ and the condition (a-iii), the fiber of the spherical Milnor fibration of $f$ 
is homotopy equivalent to 
a space obtained from $(f_{1}, f_{2})^{-1}(b_{g})$ 
by gluing to $(f_{1}, f_{2})^{-1}(\ast)$ $n_{1}$ copies of $\tilde{F}_{1}$ and $n_{2}$ copies of 
$\tilde{F}_{2}$. 
Since $F_{j}$ is diffeomorphic to $\overline{F}_{j}$, 
$\tilde{F}_{1}$ and $\tilde{F}_{2}$ are homotopy equivalent to $\overline{F}_{2}$ and $\overline{F}_{1}$ 
respectively. 
This completes the proof. 
\end{proof}

\section{Zeta functions of monodromies of Milnor fibrations}
We assume that a real analytic germ $P : (\Bbb{R}^{2n}, \mathbf{0}_{2n}) \rightarrow (\Bbb{R}^{2}, \mathbf{0}_{2})$ 
satisfies the conditions (a-i) and (a-ii). 
Let $F_{P}$ be the fiber of the Milnor fibration of $P$. 
Set $P_{j}(\lambda) = \det (\text{Id} - \lambda h_{*, j})$, where 
$h_{*, j} : H_{j}(F_{P}, \Bbb{C}) \rightarrow H_{j}(F_{P}, \Bbb{C})$ is an isomorphism induced by 
the monodromy of $P$ for $j \geq 0$. 
Then the \textit{zeta function $\zeta(\lambda)$ of the monodromy} is defined by 
\[
\zeta(\lambda) = \textstyle \prod_{j=0}^{2n-2}P_{j}(\lambda)^{(-1)^{j+1}}.
\]
See \cite[Section 9]{M2} and \cite[Chapter I]{O0}. 

In this section, we study the zeta function of the monodromy of $f$, 
where $f$ is a real analytic germ in Theorem~$3$. 
Let $g$ be a $2$-variable strongly non-degenerate mixed polynomial which 
is locally tame along vanishing coordinate subspaces. 
We denote $D = \{ (z_{1}, z_{2}) \mid z_{1}z_{2} = 0\} \subset D^{2}_{\delta} \times D^{2}_{\delta}$. 
Take a sufficiently small positive real numbers $\delta$ and $\tilde{\delta}$ such that 
and $\tilde{\delta} \ll \delta$. 
Consider the following pairs of maps 
\begin{equation*}
\begin{split}
&g : (B^{4}_{\delta} \cap g^{-1}(B_{\tilde{\delta}}), 
D \cap g^{-1}(B^{2}_{\tilde{\delta}})) \rightarrow \partial B^{2}_{\tilde{\delta}}, \\
&g/\lvert g\rvert : (\partial B^{4}_{\delta} \setminus g^{-1}(0), 
(\partial B^{4}_{\delta} \cap D) \setminus g^{-1}(0)) \rightarrow S^{1}. 
\end{split}
\end{equation*}
By \cite[Theorem 10]{O2}, Oka showed that the spherical and the tubular Milnor fibrations of $g$ are equivalent. 
Since $g$ satisfies $a_{f}$-condition and $\tilde{\delta}$ is sufficiently small, 
the fibers of two maps meet transversely spheres and $D$. 
So the above maps 
are locally trivial fibrations. Moreover the fibrations are fiber isomorphic.

Let $\ast$ be a point of $F_{g} \setminus D$, 
where $F_{g}$ is the Milnor fiber of $g$. 
By using the above fibrations, we have the sequence of groups: 
\[
1 \rightarrow \pi_{1}(F_{g} \setminus D, \ast) \xrightarrow{i_{\ast}} 
\pi_{1}(S^{3}_{\delta} \setminus (g^{-1}(0) \cup D), \ast) 
\xrightarrow{(g/\lvert g\rvert)_{\ast}} \Bbb{Z} \rightarrow 1, 
\]
where $i$ is an inclusion 
$F_{g} \setminus D \hookrightarrow S^{3}_{\delta} \setminus (g^{-1}(0) \cup D$). 
Consider 
\[ 
A^{q} = H^{q}(F_{1} \times F_{2}, \Bbb{C}), \ \ 
G = \pi_{1}(S^{3}_{\delta} \setminus (g^{-1}(0) \cup D), \ast), \ \ 
H = \pi_{1}(F_{g} \setminus D, \ast)
\] 
for $q \geq 0$. 
Since the restricted map $(f_{1} \times f_{2}) : 
(B^{n}_{\varepsilon} \times B^{m}_{\varepsilon}) \cap (f_{1} \times f_{2})^{-1}(B^{4}_{\delta} \setminus D) 
\rightarrow B^{4}_{\delta} \setminus D$ is a locally trivial fibration, we have a monodromy representation 
\[
\rho : \pi_{1}(B^{4}_{\delta} \setminus D) = 
\pi_{1}(S^{3}_{\delta} \setminus  D) = \Bbb{Z}^{2} \rightarrow 
\mathbf{Aut}(A^{q}). 
\]
The generators of $\Bbb{Z}^{2}$ are chosen so that 
$(1, 0)$ and $(0, 1)$ are meridians of the link components $\{z_{1} = 0\}$ and 
$\{z_{2} = 0\}$ respectively. 
By the inclusion $S^{3}_{\delta} \setminus (g^{-1}(0) \cup D) \hookrightarrow 
S^{3}_{\delta} \setminus D$, 
$A^{q}$ becomes 
a $G$-module, and by $i_{\ast}$ a $H$-module. 
Set 
\begin{equation*}
\begin{split}
\text{Der}(H, A^{q}) &= 
\{ d :H \rightarrow A^{q} \mid d(h_{1}h_{2}) = d(h_{1}) + d(h_{1})h_{2} \text{ for all } h_{1}, h_{2} \in H \}, \\
H^{0}(H, A^{q}) &= (A^{q})^{H}, \ \  
H^{1}(H, A^{q}) = \text{Der}(H, A^{q})/\text{Im}\>\delta, 
\end{split}
\end{equation*} 
where $\delta : A^{q} \rightarrow \text{Der}(H, A^{q})$ is defined by 
$\delta(a)(k) = \rho(k)(a) - a$ for $a \in A^{q}$ and $k \in H$. 
Let $h \in G$ such that $(g/\lvert g\rvert)_{\ast}(h) = 1$. 
The automorphism $c_{h} : H \rightarrow H$ is defined by 
$c_{h}(k) = h^{-1}kh$ for $k \in H$. 
Then the maps $\rho(h) : A^{q} \rightarrow A^{q}$ and 
$c_{h} : H \rightarrow H$ induces an automorphism of the exact sequence of 
$\Bbb{C}$-vector spaces: 
\def\mapright#1{\smash{\mathop{\longrightarrow}\limits^{{#1}}}}
\def\mapdown#1{\Big\downarrow\rlap{$\vcenter{\hbox{$#1$}}$}}
\[
\begin{matrix}
0 \rightarrow H^{0}(H, A^{q}) &\rightarrow A^{q} &\xrightarrow{\delta} \text{Der}(H, A^{q}) &\rightarrow H^{1}(H, A^{q}) &\rightarrow 0 \\
\mapdown{h_{0}^{\ast}}& \mapdown{\rho({h})}& \mapdown{h_{\text{Der}}}& \mapdown{h_{1}^{\ast}}& \\
0 \rightarrow H^{0}(H, A^{q}) &\rightarrow A^{q} &\xrightarrow{\delta} \text{Der}(H, A^{q}) &\rightarrow H^{1}(H, A^{q}) &\rightarrow 0, 
\end{matrix}
\]
where 
$h_{j}^{\ast} : H^{j}(H, A^{q}) \rightarrow H^{j}(H, A^{q})$ is the automorphism induced by 
$\rho(h)$ and $c_{h}$ for $j = 1, 2$. 

Set $\Delta_{h}(\lambda) = \det(1 - \lambda \rho(h))$ 
and $\Delta_{\text{Der}}(\lambda) = \det(1 - \lambda h_{\text{Der}})$. 
Then by the above exact sequence, 
$h_{\text{Der}}$ is determined by $\rho(h)$ and $c_{h}$. 
So we define 
\begin{equation*}
\begin{split}
(\zeta_{g, D}(\lambda))_{q}^{(-1)^{q}} &= \det(1 - \lambda h_{0}^{\ast})/\det(1 - \lambda h_{1}^{\ast}) \\
                                   &= \Delta_{h}(\lambda)/\Delta_{\text{Der}}(\lambda). 
\end{split}
\end{equation*}
The automorphisms $h_{0}^{\ast}$ and $h_{1}^{\ast}$ do not depend on the choice of $h$. 
See \cite[p. 116]{Se}. 
Thus $(\zeta_{g, D})_{q}$ is well-defined.


Note that $H$ is a free group. 
Let $b_{1}, \dots, b_{\mu(g,D)}$ be generators of $H$. 
By using the map 
\[
\text{Der}(H, A^{q}) \rightarrow (A^{q})^{\mu(g,D)}, 
\delta \mapsto (\delta(b_{1}), \dots, \delta(b_{\mu(g,D)})), 
\]
$\text{Der}(H, A^{q})$ can be identified with $(A^{q})^{\mu(g,D)}$. 

Let $\tilde{i}_{\ast} : \Bbb{Z}[H] \rightarrow \Bbb{Z}[G]$ be
the homomorphism induced by $i_{\ast} $ and 
$\tilde{\rho} : \Bbb{Z}[G] \rightarrow \Bbb{Z}[\mathbf{Aut} A^{q}]$ be
the homomorphism induced by $\rho$. 
The homomorphism of rings 
$s : \Bbb{Z}[\mathbf{Aut} A^{q}] \rightarrow \mathbf{End}_{\Bbb{C}}A^{q}$ is defined by 
\[
s(\textstyle\sum_{i}c_{i}[a^{i}_{jk}]) = [\textstyle\sum_{i}c_{i}a^{i}_{jk}]_{jk}. 
\]
Let $\frac{\partial}{\partial b_{j}} : \Bbb{Z}[H] \rightarrow \Bbb{Z}[H]$ be the derivation 
determined by $\frac{\partial b_{i}}{\partial b_{j}} = \delta_{ij}$ 
for $1 \leq i, j \leq \mu(g,D)$. 
We denote $c_{h}(b_{i})$ by $w_{i}$. 
Note that $h_{\text{Der}}$ is determined by $\rho(h)$ and $c_{h}$. 
By using the derivation rule, we have 
\[
[h_{\text{Der}}] = 
\Bigl[s \circ \tilde{\rho}\Bigl(h \cdot \tilde{i}_{\ast}\Bigl(\frac{\partial w_{i}}{\partial b_{j}}\Bigr)\Bigr)\Bigr]_{ij}. 
\]
See \cite[p. 73]{N1}. 
We set $K_{j} = S^{3}_{\delta} \cap \{z_{j} = 0\}$ for $j = 1, 2$. 
Consider the multilink 
\[
(S^{3}_{\delta}, S^{3}_{\delta} \cap g^{-1}(0)) 
= (S^{3}_{\delta}, m_{1}K_{1} \cup m_{2}K_{2} \cup m_{3}K_{3} \cup \cdots \cup m_{r}K_{r}), 
\]
where $K_{j}$ is an oriented knot and 
$m_{j} \in \Bbb{Z}$ for $ 1 \leq j \leq r$. 
Note that $m_{j} = 0$ if and only if $g|_{z_{j}=0} \not\equiv 0$ for $1 \leq j \leq 2$. 
Put $L = (S^{3}_{\delta}, K_{1} \cup K_{2} \cup K_{3} \cup \cdots \cup K_{r})$. 
Then $(\zeta_{h, D})_{q}$ can be calculated by the Alexander polynomial of $L$ \cite{N1}. 
We can show the following theorem. 

\begin{theorem}
Let $f_{1}, f_{2}$ and $g$ be real analytic germs in Theorem $3$. 
Let $H_{j, k} : H_{k}(F_{j}, \Bbb{C}) \rightarrow H_{k}(F_{j}, \Bbb{C})$ 
be the monodromy matrix induced by the monodromy of $f_{j}$ for $j = 1, 2$ and $k \geq 0$. 
Set $E_{q, 1} = \bigoplus_{i + j = q}(H_{1, i}) \otimes (I_{2})_{j}$ and 
$E_{q, 2} = \bigoplus_{i + j = q}(I_{1})_{i} \otimes (H_{2, j})$, 
where 
$(I_{l})_{k} : H_{k}(F_{l}, \Bbb{C}) \rightarrow H_{k}(F_{l}, \Bbb{C})$ is the 
identity matrix for $l = 1, 2$ and $k \geq 0$. 
Then 
up to multiplication by monomials~$\pm \lambda^{u}$,
the zeta function of $f = g(f_{1}, f_{2})$ is determined by 
\[
\zeta_{f}(\lambda) = \zeta_{f_{1}}(\lambda^{n_{2}})\zeta_{f_{2}}(\lambda^{n_{1}})
\textstyle \prod_{q} \det \Delta_{L}(\lambda^{m_{1}}E_{q,1}, \lambda^{m_{2}}E_{q,2}, 
\lambda^{m_{3}}I, \dots, \lambda^{m_{r}}I)^{(-1)^q}, 
\]
where $\Delta_{L}(\lambda_{1}, \dots, \lambda_{r})$ is the Alexander polynomial of $L$. 
If $n_{j} = 0$, then we set $\zeta_{f_{(j+1)\bmod 2}}(\lambda^{n_{j}}) = 1$ for $j = 1, 2$. 
\end{theorem}

\begin{remark}
Let $L_{j}$ be the link obtained form $L$ by reversing the orientation of $K_{j}$. 
Then two Alexander polynomials satisfy 
\[
\Delta_{L}(\lambda_{1}, \dots, \lambda_{r}) = 
\epsilon \lambda^{u'}_{j}\Delta_{L_{j}}(\lambda_{1}, \dots, \lambda^{-1}_{j}, \dots \lambda_{r}), 
\]
where $\epsilon = \pm 1$ and $u' \in \Bbb{Z}$. 
We denote the link reversing orientation of $K_{j}$ by $-K_{j}$. 
Then the associated multiplicity of $-K_{j}$ is $-m_{j}$. 
Thus up to multiplication by monomials, we have 
\begin{equation*}
\begin{split}
&\det \Delta_{L}(\lambda^{m_{1}}E_{q,1}, \lambda^{m_{2}}E_{q,2}, 
\lambda^{m_{3}}I, \dots, \lambda^{m_{j}}I, \dots   \lambda^{m_{r}}I) \\
= &\det \Delta_{L_{j}}(\lambda^{m_{1}}E_{q,1}, \lambda^{m_{2}}E_{q,2}, 
\lambda^{m_{3}}I, \dots, \lambda^{-m_{j}}I, \dots, \lambda^{m_{r}}I)
\end{split}
\end{equation*}
for any $q \geq 0$. 
\end{remark}

\begin{proof}[Proof of Theorem $5$]
Let $Q_{\lvert \tilde{\delta} \rvert} = g^{-1}(\partial D^{2}_{\tilde{\delta}}) \cap B^{4}_{\delta}$ and 
$E_{\lvert \tilde{\delta} \rvert} = (f_{1}, f_{2})^{-1}(Q_{\lvert \tilde{\delta} \rvert}) \cap B^{n+m}_{\varepsilon}$. 
By using Theorem $3$, 
we can show that there exists a deformation retract 
\[
\Phi_{t} : f^{-1}(\partial D^{2}_{\tilde{\delta}}) \cap B^{n+m}_{\varepsilon} \rightarrow 
f^{-1}(\partial D^{2}_{\tilde{\delta}}) \cap B^{n+m}_{\varepsilon}
\]
such that $\Phi_{0} = \text{id}, 
\text{Im}\>{\Phi}_{1} = E_{\lvert \tilde{\delta} \rvert}$ 
and $\Phi_{t}$ 
preserves the fiber of $f$ for $0 \leq t \leq 1$. 
Then $f|_{E_{\lvert \tilde{\delta} \rvert}}$ and  
$f|_{f^{-1}(\partial D^{2}_{\tilde{\delta}}) \cap B^{n+m}_{\varepsilon}}$ 
satisfy 
\[
f|_{E_{\lvert \tilde{\delta} \rvert}} = g \circ (f_{1}, f_{2})|_{E_{\lvert \tilde{\delta} \rvert}} = 
f|_{f^{-1}(\partial D^{2}_{\tilde{\delta}}) \cap B^{n+m}_{\varepsilon}} \circ \Phi_{1}|_{E_{\lvert \tilde{\delta} \rvert}} . 
\]
So the monodromy of $f$ can be identified to the monodromy of $f|_{E_{\lvert \tilde{\delta} \rvert}}$. 
Since $(f_{1}, f_{2})$ satisfies $a_{f}$-condition, 
by Thom's second isotopy lemma, 
\[
(f_{1}, f_{2}) : E_{\lvert \tilde{\delta} \rvert} \rightarrow Q_{\lvert \tilde{\delta} \rvert}
\] 
is a locally trivial fibration. 
Thus 
the monodromy of $f$ can be regarded as a lifting of the monodromy of $g$.

Let $T$ be a small neighborhood of $Q_{\lvert \tilde{\delta} \rvert} \cap D$ 
in $Q_{\lvert \tilde{\delta} \rvert}$. 
We consider the following subspace of $E_{\lvert \tilde{\delta} \rvert}$: 
\begin{equation*}
\begin{split} 
&\mathcal{T} = (f_{1}, f_{2})^{-1}(T) \cap B^{n+m}_{\varepsilon}, \ \ \
\partial \mathcal{T} = (f_{1}, f_{2})^{-1}(\partial T) \cap B^{n+m}_{\varepsilon}, \\
&\mathcal{CT} = (f_{1}, f_{2})^{-1}(Q_{\lvert \tilde{\delta} \rvert} \setminus \text{Int } T) \cap B^{n+m}_{\varepsilon}.
\end{split}
\end{equation*}
Set $\mathcal{F}_{\tilde{\delta}} = 
(f_{1}, f_{2})^{-1}(g^{-1}(\tilde{\delta}) \cap B^{4}_{\delta}) \cap B^{n+m}_{\varepsilon}$. 
Note that $E_{\lvert \tilde{\delta} \rvert} = \mathcal{CT} \cup \mathcal{T}$ and 
$\partial \mathcal{T} = \mathcal{CT} \cap \mathcal{T}$. 
We denote the corresponding zeta functions of the monodromy restricted to 
$\mathcal{T} \cap \mathcal{F}_{\tilde{\delta}}, \partial \mathcal{T} \cap \mathcal{F}_{\tilde{\delta}}$ and 
$\mathcal{CT} \cap \mathcal{F}_{\tilde{\delta}}$
by 
$\zeta_{\mathcal{T}}, \zeta_{\mathcal{CT}}$ and 
$\zeta_{\partial \mathcal{T}}$ respectively. 
Since the monodromy of $f|_{E_{\lvert \tilde{\delta} \rvert}}$ is a lifting of the monodromy of~$g$ 
which preserves $T, \partial T$ and $Q_{\lvert \tilde{\delta} \rvert} \setminus \text{Int } T$, 
we have
\[
\zeta_{f} = \zeta_{\mathcal{T}} \cdot \zeta_{\mathcal{CT}} \cdot (\zeta_{\partial \mathcal{T}})^{-1}.
\]

Since $\mathcal{CT} \cap \mathcal{F}_{\tilde{\delta}}$ is the total space of a fiber bundle 
over the path-connected base space $F_{g} \setminus \text{Int } T$ and with a fiber 
$F_{1} \times F_{2}$, 
there is a convergent $E_{2}$ cohomology spectral sequence of bigraded algebras with 
$E_{2}^{p, q} = H^{p}(F_{g} \setminus D, A^{q})$ 
and $E_{\infty}$ the bigraded algebra associated to some filtration of 
$H^{\ast}(\mathcal{CT} \cap \mathcal{F}_{\tilde{\delta}}, \Bbb{C})$ 
\cite[p. 498]{Sp}. 
Note that $F_{g} \setminus D$ is a $K(H, 1)$-space. 
If $p$ is greater than $1$, 
$H^{p}(F_{g} \setminus D, A^{q}) =
H^{p}(H, A^{q})$ is trivial \cite[p. 220]{Su}. 
We identify $H^{p}(F_{g} \setminus D, A^{q})$ with $H^{p}(H, A^{q})$ 
for $p = 0, 1$ and $q \geq 0$. 
Then the monodromy action on $H^{p}(H, A^{q})$ can be identified to 
the action in the above diagram for $p = 0, 1$ and $q \geq 0$. 
Thus $E_{2}^{p, q} = E_{\infty}^{p, q}$ and 
$\zeta_{\mathcal{CT}}$ is equal to $\prod_{q}(\zeta_{g, D})_{q}$.

By conic structures of $f_{1}^{-1}(\mathbf{0}_{2}) \cap B^{n}_{\varepsilon}$ and 
$f_{2}^{-1}(\mathbf{0}_{2}) \cap B^{m}_{\varepsilon}$, 
$(f_{1}, f_{2})^{-1}(F_{g} \cap D)$ is homotopy equivalent to 
\[
\{ \text{$n_{1}$-copies of } F_{1}\} \cup \{ \text{$n_{2}$-copies of } F_{2}\}. 
\]
Set $\tilde{H}_{j, q}$ is the monodromy of $n_{j}$-copies of $H_{q}(F_{j}, \Bbb{C})$ 
for $j = 1, 2$ and $q \geq 0$. 
The monodromy of $n_{j}$-copies of $H_{q}(F_{j}, \Bbb{C})$ 
can be identified with the cyclic permutation of $H_{j, q}$. 
Thus $\tilde{H}_{j, q}$ is equal to 
\[
\tilde{H}_{j, q} = 
\left(
        \begin{array}{@{\,}cccccccc@{\,}}
O & \ldots & \ldots & O & H_{j, q} \\
H_{j, q} & O & \ldots & \ldots & O \\ 
O & \ddots & \ddots & \ddots & \vdots \\
\vdots & \ddots & \ddots & \ddots & \vdots \\
O & \ldots & O & H_{j, q} & O
\end{array}
\right).
\]
By \cite[p. 68]{N1}, 
$\zeta_{\mathcal{T}}(\lambda) = \zeta_{f_{1}}(\lambda^{n_{2}})\zeta_{f_{2}}(\lambda^{n_{1}})$. 
Note that 
$\zeta_{\partial \mathcal{T}}$ is trivial. See \cite[p. 91]{N1}. 
Thus we obtain that 
\[
\zeta_{f}(\lambda) = \zeta_{f_{1}}(\lambda^{n_{2}})\zeta_{f_{2}}(\lambda^{n_{1}})\textstyle \prod_{q}(\zeta_{g, D})_{q}. 
\]
Since $(\zeta_{g, D})_{q}$ is determined by $h_{\text{Der}}$ and $\rho(h)$, 
by using the similarly argument of the proof of \cite[Theorem C]{N1}, 
$\Delta_{h}(\lambda)/\Delta_{\text{Der}}(\lambda)$ can be represented by 
the Alexander polynomial of $L$. 
Thus we have 
\[
(\zeta_{g, D})^{(-1)^{q}}_{q} = 
\det \Delta_{L}(\lambda^{m_{1}}E_{q,1}, \lambda^{m_{2}}E_{q,2}, 
\lambda^{m_{3}}I, \dots, \lambda^{m_{r}}I). 
\]
This completes the proof. 
\end{proof}

\begin{example}
Set $g = z_{1}z_{2}\prod_{j=1}^{k}(z_{1}^{p_{1}} + \alpha_{j}z_{2}^{p_{2}})
\prod_{j=k+1}^{k + \ell}\overline{(z_{1}^{p_{1}} + \alpha_{j}z_{2}^{p_{2}})}$. Assume that 
$\alpha_{j} \neq \alpha_{j'}$for $j \neq j'$ and $1 \leq j, j' \leq k + \ell$. 
Then $g$ is a $2$-variable strongly non-degenerate mixed polynomial which 
is locally tame along vanishing coordinate subspaces. 
By \cite{EN}, the Alexander polynomial $\Delta_{L}(\lambda_{1}, \dots, \lambda_{k+\ell + 2})$ 
is equal to 
$(\lambda_{1}^{p_{2}}\lambda_{2}^{p_{1}}\lambda_{3}^{p_{1}p_{2}}\cdots \lambda_{k+\ell+2}^{p_{1}p_{2}} - 1)^{k+\ell}$. 
Therefore the zeta function of the monodromy of $f$ is given by 
\begin{equation*}
\begin{split}
\zeta_{f}(\lambda) &= 
\textstyle \prod_{q} \det \Delta_{L}(\lambda^{m_{1}}E_{q,1}, \lambda^{m_{2}}E_{q,2}, 
\lambda^{m_{3}}I, \dots, \lambda^{m_{k+\ell+2}}I)^{(-1)^q}        \\
                   &= 
\textstyle \prod_{i, j} \det (\lambda^{p_{1} + p_{2} + p_{1}p_{2}(k-\ell)}
(H_{1, i})^{p_{2}} \otimes (H_{2,j})^{p_{1}} - I)^{(-1)^{i+j}(k+\ell)}. 
\end{split}
\end{equation*}
\end{example}

\begin{example}
Set $f_{1} : \Bbb{R}^{3} \rightarrow \Bbb{R}^{2}, f(x_{1}, x_{2}, x_{3}) = 
(x_{3}(x_{1}^{2} + x_{2}^{2} + x_{3}^{2}), x_{2} - x_{1}^{3}), 
f_{2} : \Bbb{C} \rightarrow \Bbb{C}, f_{2}(w) = w^{2}$ and 
$g : \Bbb{C}^{2} \rightarrow \Bbb{C}, g(z_{1}, z_{2}) = z_{1}^{2} + z_{2}^{3}$. 
Let $f : \Bbb{R}^{3} \times \Bbb{C} \rightarrow \Bbb{R}^{2}$ be a real analytic map which is defined by 
$f(x_{1}, x_{2}, x_{3}, w) = g(f_{1}, f_{2})(x_{1}, x_{2}, x_{3}, w)$. 
By \cite{AT}, $f_{1}$ has an isolated singularity at the origin. 
Hence $f$ also satisfies the conditions (a-i) and (a-ii). 
Note that $\zeta_{f_{1}}(\lambda) = \frac{1}{\lambda - 1}, 
\zeta_{f_{2}}(\lambda) = \frac{1}{\lambda^{2} - 1}$ and 
$\det \Delta_{L}(\lambda_{1}, \lambda_{2}, \lambda_{3}) = 
\lambda_{1}^{3}\lambda_{2}^{2}\lambda_{3}^{6} - 1$. 
By Theorem $5$, $\zeta_{f}(\lambda)$ is equal to 
\[
\frac{1}{(\lambda^{2} - 1)(\lambda^{6} - 1)}
\det \begin{pmatrix}
     \lambda^{6} - 1 & 0 \\
     0 & \lambda^{6} - 1
     \end{pmatrix} 
     = \frac{\lambda^{6} - 1}{\lambda^{2} - 1} = \lambda^{4} + \lambda^{2} + 1. 
\]

\end{example}


\end{document}